\documentclass[a4,11pt]{article}
\usepackage[colorlinks,urlcolor=blue,linkcolor=blue,citecolor=blue]{hyperref}
\usepackage{color,array}
\usepackage{amsmath}
\usepackage{amsthm}
\usepackage{graphicx}
\usepackage{caption} 
\usepackage{booktabs} 
\usepackage{color}

\usepackage{array} 
\usepackage{graphicx}
\usepackage{graphicx,epstopdf}

\usepackage{epsfig,epsf,epstopdf}
\usepackage{relsize}
\usepackage[pdf]{pstricks}
\usepackage{mathptmx}
\usepackage{amsmath}
\usepackage{amssymb}
\usepackage{amsfonts}
\usepackage{mathtools}
\usepackage{euscript}
\usepackage{textcomp}
\usepackage{tikz, pgfplots}
\pgfplotsset{compat=1.18}
\usetikzlibrary{shapes,arrows}
\usetikzlibrary{shapes, positioning, arrows.meta}
\usepackage{fancyhdr}
\usetikzlibrary{shadows,calc}
\usetikzlibrary{positioning}
\newtheorem{proposition}{Proposition}[section]
\newtheorem{corollary}{Corollary}[section]
\newtheorem{conjecture}{Conjecture}[section]
\makeatletter
\def\ps@pprintTitle{%
	\let\@oddhead\@empty
	\let\@evenhead\@empty
	\let\@oddfoot\@empty
	\let\@evenfoot\@oddfoot
}
\makeatother
\usepackage{amsmath,amssymb,amsfonts}
\usepackage{algorithmic}
\newtheorem{theorem}{Theorem}
\newtheorem{lemma}{Lemma}
\usepackage[a4paper, total={5.25in, 8in}]{geometry}
\newtheorem{example}{Example}[section]

\newtheorem{remark}{Remark}[section]

\newtheorem{definition}{Definition}[section]
\usepackage{authblk}
\begin{document}

\title{\textbf{Affirmative Results on a Conjecture on the Column Space of the Adjacency Matrix}}

\author{
\textbf{S.~Akansha}$^{1}$\footnote{Part of this work was done when the author was at Manipal Institute of Technology, Manipal, India.}  ~and
\textbf{K.~C.~Sivakumar}$^{2}$\thanks{Corresponding author.}\\
Department of Mathematics\\ 
$^{1}$The LNM Institute of Information Technology, Jaipur 302031, India\\
$^{2}$Indian Institute of Technology Madras, Chennai 600036, India\\
\texttt{agrawal.akansha5@gmail.com}, \texttt{kcskumar@iitm.ac.in}
}
\date{}
\maketitle
\begin{abstract}
The Akbari-Cameron-Khosrovshahi (ACK) conjecture, which appears to be unresolved, states that for any simple graph $G$ with at least one edge, there exists a nonzero {$\{0,1\}$}-vector in the row space of its adjacency matrix that is not a row of the matrix itself. In this talk, we present a unified framework that includes several families and operations of graphs that satisfy the ACK conjecture. 
Using these fundamental results, we introduce new graph constructions and demonstrate, through graph structural and linear  algebraic arguments, that these constructions adhere to the conjecture. Further, we show that certain graph operations preserve the ACK property. 
These results collectively expand the known classes of graphs satisfying the conjecture and provide insight into its structural invariance under composition and extension. 
\end{abstract}



\section{Introduction}
Let us recall what has come to be known as the ACK Conjecture. All graphs in this article are simple and connected.
\begin{conjecture}\label{con:ack1}
  \textbf{\cite[Question 2]{akbari2004ranks}} For any graph $G$ (containing at least one edge), there exists a nonzero $\{0,1\}$-vector in the row space of its adjacency matrix $A_G$, over the field $\mathbb{R},$ that is not one of the rows of $A_G$.  
\end{conjecture}
Previous studies have established Conjecture~\ref{con:ack1} for nonsingular graphs, that is, for graphs whose adjacency matrices are invertible \cite{bera2022conjecture, sciriha2025potential}. Consequently, the conjecture remains unresolved only for singular graphs. Within the class of singular graphs, the conjecture has been verified for graphs of diameter at least $4$, as well as for graphs of diameter $2$ that have no dominating vertex and exactly $2n-5$ edges, where $n$ denotes the order of the graph \cite{bera2022conjecture}. More recently, Sciriha et al.~\cite{sciriha2025potential} introduced a novel approach based on changes in the nullity of the adjacency matrix under vertex addition. Using this framework, they identified a restricted class of singular graphs that may serve as potential counter-examples and showed, in particular, that such graphs cannot contain pendant vertices.

In this paper, we focus on identifying and characterizing additional families of singular graphs, distinct from the known cases, for which the ACK conjecture holds. We introduce the concept of kernel-vector-based non-duplicate zero-sum subsets as an equivalent formulation of the ACK conjecture. Our main contributions include the identification of several new families of singular graphs and graph operations that satisfy the ACK conjecture. First, we introduce a class of diameter-$2$ graphs with a dominating vertex that satisfy all known necessary conditions for potential counter-examples, as introduced by Sciriha et al.~\cite{sciriha2025potential}, and we verify that the conjecture holds for this class. Second, we establish that the ACK property is invariant under specific graph operations. We also present a constructive approach for building graphs of arbitrary order that satisfy the conjecture via vertex additions associated with disjoint zero-sum subsets.

The remainder of this paper is organized as follows. In Section \ref{sec:prel}, we establish the necessary notation and recall fundamental concepts including core and nut graphs. In Section \ref{secreform}, we introduce the pivotal concept of zero-sum subsets and use it as an equivalent formulation for the ACK conjecture. Section \ref{sec:graphsinC} is devoted to the construction of new families of singular graphs lying in the class of potential counter examples, and we prove that these families satisfy the ACK conjecture by identifying their underlying zero-sum structures. These are presented in Theorem \ref{thm:ackholdsforsn} and Corollary \ref{acknut}. In Section \ref{sec:graphopertn}, we study graph operations that preserve the ACK property, demonstrating invariance under cartesian products with $K_2$ (Theorem \ref{thm:concart}) and specific vertex-addition procedures (Theorem \ref{thm:single_vertex_addition}). Finally, Section~\ref{sec:coreack} identifies broader classes of core graphs of arbitrary nullity that satisfy the conjecture. These are presented in Theorem \ref{firstcore} and Theorem \ref{secondcore}.

\section{Preliminaries}\label{sec:prel}

The vector ${\bf e}$ will denote the vector each of whose entries is $1.$ The dimension of this vector will be clear from the context. A {matrix will be referred to as a {\it full} matrix, if all its components are nonzero. A similar definition applies to a vector.} For a matrix $X$, we will denote its $i$-th column by $x^i$. In what follows, we recall some basic definitions that are required in our discussion. We also include some notation.

For a graph $G=(V,E)$, where $V=\{v_1,v_2, \ldots, v_n\}$ is the set of vertices, and $E$ is the edge set, we use $(a_{ij})=A_{G}$ to denote its adjacency matrix. Thus, $a_{ij}=a_{ji}$ for all $i,j$ and $a_{ij} \neq 0,$ if there is an edge joining the vertices $v_i$ and $v_j$, with $i \neq j.$ For the matrix $X$, we let $N(X), R(X)$ to denote the null space and the range space of $X$, respectively.

\begin{enumerate}
     \item A vertex \(v\) in a graph $G$ is called a {\it core vertex} (CV), if there is a kernel vector $x$ such that $x_v$, the entry of $x$ corresponding to $v,$ is nonzero. A graph is a {\it core graph}, if each of its vertices is a core vertex (CV). Paraphrasing, a core graph is a graph which has a kernel vector, each of whose entries is nonzero.
     \item A graph \(G\) is called a {\it nut graph} if \(A_G\) has nullity $1,$ and has the property that all the components of any nonzero vector in $N(A_G)$, are nonzero. In other words, a nut graph is a core graph with nullity $1$.
The smallest nut graph have seven vertices and there are precisely three of them. Also, nut graphs are non-bipartite and
have no leaves. Here are two pertinent results on nut graphs: 
\begin{enumerate}
    \item A graph $G$ is a nut graph iff $\rm det (A) =0$ and all the entries of the matrix $\rm adj (A)$ are nonzero \cite[Lemma 2.2]{cfg}. The following is an example of a core graph which is not a nut graph \cite{sciriha2007characterization}.
\begin{figure}[ht]
\begin{center}
\begin{tikzpicture}
    \tikzstyle{v} = [circle, fill=black, inner sep=1.5pt, outer sep=0pt]

    \node[v] (m1) at (0, 0) {};
    \node[v] (m2) at (2, 0) {};
    
    \node[v] (tl) at (-1, 1) {};
    \node[v] (bl) at (-1, -1) {};

    \node[v] (tr) at (3, 1) {};
    \node[v] (br) at (3, -1) {};

    \draw (tl) -- (bl) -- (m1) -- (tl); 
    \draw (tr) -- (br) -- (m2) -- (tr); 
    \draw (m1) -- (m2);                 
    \draw (tl) -- (tr);
    \draw (bl) -- (br);
    
\end{tikzpicture}
\caption{Core graph which is not a nut graph.}
\label{fig:core}
\end{center}
\end{figure}
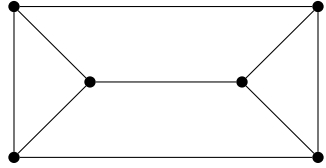
\item\label{thm:nut_construction} 
Let $G$ be a graph obtained by adjoining a new vertex to a graph $H$ with a nonsingular adjacency matrix $A_H$ so that this vertex is adjacent to exactly two distinct vertices $v_i, v_j \in V(H)$. Then $G$ is a nut graph iff for $(b_{ij})=B:=A^{-1}_H$ the following two conditions hold \cite[Theorem 2.1]{cfg}:
        \begin{enumerate}
            \item $b_{ii} + b_{jj} + 2b_{ij} = 0$.
            \item {$b^i+b^j$ is a full vector.}
        \end{enumerate}
        The following is an example of such a nut graph:
 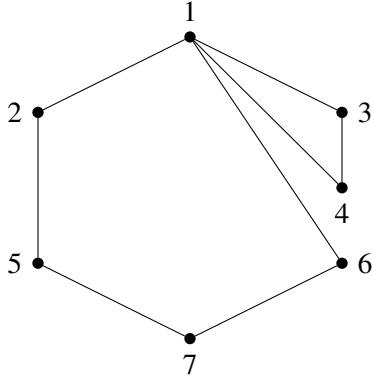
\begin{figure}[ht]
\centering
\begin{tikzpicture}
    \tikzstyle{v} = [circle, fill=black, inner sep=1.5pt, outer sep=0pt]

    \node[v,label=above:$1$] (1) at (0,0) {};
    \node[v,label=left:$2$]  (2) at (-2,-1) {};
    \node[v,label=right:$3$] (3) at (2,-1) {};
    \node[v,label=below:$4$] (4) at (2,-2) {};
    \node[v,label=left:$5$]  (5) at (-2,-3) {};
    \node[v,label=right:$6$] (6) at (2,-3) {};
    \node[v,label=below:$7$] (7) at (0,-4) {};

    \draw (1)--(2);
    \draw (1)--(3);
    \draw (1)--(4);
    \draw (1)--(6);
    \draw (2)--(5);
    \draw (3)--(4);

    \draw (5)--(7);
    \draw (6)--(7);
\end{tikzpicture}
\caption{A 7-vertex nut graph constructed from base graph $H$ by connecting a new vertex $7$ to vertices $5$ and $6$.}
\label{fig:nut_construction}
\end{figure}       
\end{enumerate}

     \item An eigenvalue of a matrix is {\it main}, if an associated eigenvector is not orthogonal to $\mathbf{e}.$
     \item Given a graph $G$, we let $G+v$ denote the graph obtained from $G$ by adding a vertex $v$, in such a way that $v$ is either an isolated vertex or is adjacent to one or more vertices in $G$. We may sometimes refer to $G$ as the {\it base} graph. 
     \item The {\it nullity} of a graph $G$ is defined as the nullity of $A_G$, denoted by $\eta(G).$ Given a base graph $G$, the vertex $v$ in $G+v$ is called a {\it Parter vertex} if $\eta(G+v)=\eta(G)-1.$ (This is called a $CFV_{upp}$ vertex in \cite{sciriha2025potential}). We make use of Parter vertices to obtain an equivalent formulation of the ACK conjecture, in Proposition \ref{prop:ackequivgeneral}.
     \item For a given vertex $u$ of a graph $G$, the vector $a^u$ is the $\{0,1\}$-vector whose $r$th-component is $1$ iff the vertex $v_r$ is adjacent to $u$. We call $a^u$ the {\it adjacency vector} corresponding to $u$. In \cite{sciriha2025potential}, $a^u$ is referred to as a {\it characteristic vector.}
     \item A vertex $v$ added to a graph $G$ in such a way that the adjacent vector $a^v$ is not equal to $a^u$ for any other added vertex $u$, is called a {\it non-duplicate vector}. 
     \item For a subset $S \subseteq V$ of a graph $G$, the vector 
$\chi_S$ is the \( \{0,1\}\)-vector whose $r$th component is $1$ iff the vertex $v_r \in S$. We call $\chi_S$ the characteristic vector corresponding to the vertex subset $S$.  
\end{enumerate}

The following characterization of a Parter vertex will prove to be quite useful.

\begin{theorem} \cite[Theorem  3.3]{sciriha2025potential}\label{parterchar}
Let $G$ be a singular base graph and $v$ be a non-duplicated vertex added to $G$ such that $v$ is adjacent to at least one vertex in $G$. Then $v$ is {\textit not} a Parter vertex of $G+v$ iff $a^v \in R(A_G)$. \end{theorem}

The next result will be used in proving that the nullity of satellite graph is $1$ (Theorem \ref{sat_nut}).
\begin{lemma}\label{nul1}
Let $B$ be a singular matrix. Suppose that no nonzero vector in $N(B)$ has a zero coordinate. Then $B$ has nullity $1$.
\end{lemma}
\begin{proof}
Suppose that $\text{dim}(N(B)) \geq 2$ and let $y, z \in N(B)$ be linearly independent. Consider their first coordinates $y_1$ and $z_1$, respectively (both of which are nonzero). Set $w:=z_1 y - y_1 z.$ Then $0 \neq w \in N(B)$ has its first coordinate zero, a contradiction.
\end{proof}
Let ${\cal C}$ denote the class of all graphs $G$ which have the potential to be counter examples to the ACK-conjecture. Recently, the following necessary conditions for such graphs were obtained. 
\begin{theorem}\label{neccon}
\cite[Theorem 5.1]{sciriha2025potential}
Let $G \in {\cal C}$. We then have the following:
\begin{enumerate}
\item $G$ is a core graph;
\item $0$ is a main eigenvalue of $G$;
\item every vertex of $G$ lies on a triangle;
\item for every vertex $u \in V$, each vertex in $N(u)$ forms a triangle containing $u$;
\item $G$ is not regular;
\item $G$ must be connected;
\item $G$ is not bipartite;
\item the diameter of $G$ is $2$ or $3$.    
\end{enumerate} 
\end{theorem} 

We show that all graphs that we study here, satisfy these necessary conditions.

\section{A reformulation of the conjecture}\label{secreform}
We introduce the notion of a zero sum subset, using which we reformulate the ACK conjecture, for a class of graphs.
\begin{definition}[Zero-sum subset]
Given a graph $G$, let $0\neq x \in N(A_G)$.  
A {nonempty} subset $S\subseteq V$ is called a zero-sum subset relative to $x$, if
\[
\sum_{v\in S} x_v = 0.
\]
Such a subset will be referred to as non-duplicate, if its characteristic vector 
$\chi_S$ does not coincide with any row of $A_G$.
\end{definition}
\begin{lemma}\label{lem:singulargraphzss}
Any non-trivial, simple, connected singular graph has a non-empty zero-sum subset.
\end{lemma}
\begin{proof}
    Let $G$ be a non-trivial, simple, connected graph with vertex set $V$ such that there exists a non-zero vector $x \in N(A_G)$. Therefore, for any vertex $v \in V$, we have $\sum_{w \in N(v)} x_w = 0$. Let $v_0$ be an arbitrary vertex in $V,$ and define the set $S = N(v_0)$. By the kernel condition at $v_0$, we have $\sum_{w \in N(v_0)} x_w = 0$. Since $S=N(v_0)$, the sum over $S$ is zero. As $G$ is connected, it has no isolated vertices. Thus, $\text{deg}(v_0) \geq 1$, which implies that $S = N(v_0) \neq \emptyset$. 
\end{proof}
For graphs of nullity one, there is an equivalent formulation for the ACK conjecture. This is the next result. 
\begin{proposition}\label{prop:ackequivgeneral}
A graph $G$ satisfies the ACK conjecture if and only if there exists a non-empty,
non-duplicate subset $S\subseteq V(G)$ such that $\chi_S \in N(A_G)^{\perp}.$
\end{proposition}
\begin{proof}
The proof will be included in the revision to this version.
\end{proof}

\begin{remark}\label{rem:ackequiv}
Note that in the case of graphs with nullity one, the condition $\chi_S \in N(A_G)^{\perp}$ simply means that $S$ is a zero-sum subset with respect to a nonzero kernel vector. 
In particular, if $G$ is a graph of nullity one, then $G$ satisfies the ACK conjecture if and only if there exists a non-empty, non-duplicate zero-sum subset $S \subseteq V(G)$. We will use this zero-sum subset idea throughout when proving the ACK conjecture for graphs of nullity one.
\end{remark}

\begin{corollary}\label{cor:zssnqdegree}
Let $G$ be a graph with nullity one.  
If $G$ has a non-empty zero-sum subset $S\subseteq V$ such that 
$|S| \notin \{\deg(v) : v\in V\},$ then $G$ satisfies the ACK conjecture.
\end{corollary}

\begin{proof}
By definition, any row of $A_G$ corresponding to a vertex $v$ has exactly $\deg(v)$ ones.  
Hence, if $|S|$ is not equal to any vertex degree, the characteristic vector $\chi_S$ cannot coincide with any row of $A_G$, so $S$ is automatically non-duplicate. $S$ is the required non-empty, non-duplicate zero-sum subset. Therefore, $G$ satisfies the  ACK conjecture.
\end{proof}

\begin{remark}
In general, the converse of Corollary~\ref{cor:zssnqdegree} does not hold. 
For instance, consider the graph $G$ on $14$ vertices with edge set
\[
\begin{aligned}
E(H)=\{&
\{1,2\},\ \{1,3\},\ \{1,4\},\ \{1,6\},\ \{1,7\},\ \{1,8\},\ \{1,9\},\ \{1,10\},\ \{1,11\},\ \{1,12\},\\[3pt] &\{1,13\},\ \{1,14\},\ \{2,3\},\ \{2,10\},\ \{2,11\},\ \{3,4\},\ \{3,5\},\ \{4,5\},\ \{4,12\},\ \{5,6\},\ \{6,7\},\\[3pt]
&\{6,8\},\ \{6,13\},\ \{7,8\},\ \{7,9\},\ \{8,9\},\ \{8,10\},\ \{8,14\}, \ \{9,10\}\}.
\end{aligned}
\]
\begin{figure}[h!]
\centering
\begin{tikzpicture}[scale=1] 

    \tikzset{
        v_dom/.style = {circle, fill=blue, draw=black, minimum size=4pt, inner sep=0pt},
        v_4/.style   = {circle, fill=red, draw=black, minimum size=4pt, inner sep=0pt},
        v_2/.style   = {circle, fill=black, draw=black, minimum size=4pt, inner sep=0pt}
    }

    \node[v_dom, label=above:1] (vdom) at (0,0) {};

    \foreach \i in {1,...,9} {
        \pgfmathtruncatemacro{\labelnum}{\i+1} 
        \pgfmathtruncatemacro{\angle}{360/9 * (\i - 1) + 90} 
        \node[v_4, label=\angle:\labelnum] (u\i) at (\angle : 1.5cm) {};
    }

    \node[v_2, label={360/9 * (1 - 1) + 100}:11] (w1) at ({360/9 * (1 - 1) + 100} : 2.2cm) {};
    \node[v_2, label={360/9 * (3 - 1) + 100}:12] (w3) at ({360/9 * (3 - 1) + 100} : 2.2cm) {};
    \node[v_2, label={360/9 * (5 - 1) + 100}:13] (w5) at ({360/9 * (5 - 1) + 100} : 2.2cm) {};
    \node[v_2, label={360/9 * (7 - 1) + 100}:14] (w7) at ({360/9 * (7 - 1) + 100} : 2.2cm) {};

    \foreach \i in {1,...,9} {
        \draw[thick] (vdom) -- (u\i);
    }
    \foreach \i in {w1,w3,w5,w7} {
        \draw[thick] (vdom) -- (\i);
    }

    \draw[thick] (u1) -- (u2) -- (u3) -- (u4) -- (u5) -- (u6) -- (u7) -- (u8) -- (u9) -- (u1);

    \draw[thick] (u1) -- (w1);
    \draw[thick] (u3) -- (w3);
    \draw[thick] (u5) -- (w5);
    \draw[thick] (u7) -- (w7);
    
    \draw[thick] (u2) -- (u4);
    \draw[thick] (u5) -- (u7);
    \draw[thick] (u6) -- (u8);
    \draw[thick] (u7) -- (u9);

\end{tikzpicture}
\caption{The degree sequence is \(\{2,2,2,2, 4,4,4,4,4,4,4, 5,6, 13\}\).}
\label{fig:custom_graph_labeled_14}
\end{figure}
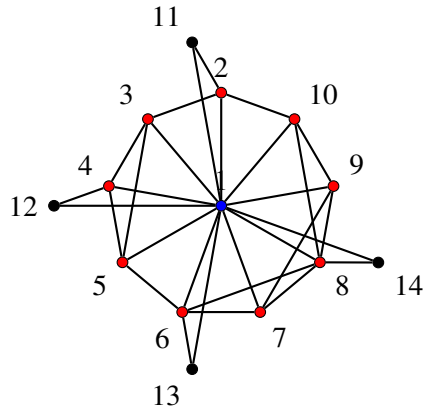

Then
\[
A_G = \left[
\begin{array}{cccccccccccccc} 
0 & 1 & 1 & 1 & 0 & 1 & 1 & 1 & 1 & 1 & 1 & 1 & 1 & 1\\
1 & 0 & 1 & 0 & 0 & 0 & 0 & 0 & 0 & 1 & 1 & 0 & 0 & 0\\
1 & 1 & 0 & 1 & 1 & 0 & 0 & 0 & 0 & 0 & 0 & 0 & 0 & 0\\
1 & 0 & 1 & 0 & 1 & 0 & 0 & 0 & 0 & 0 & 0 & 1 & 0 & 0\\
0 & 0 & 1 & 1 & 0 & 1 & 0 & 0 & 0 & 0 & 0 & 0 & 0 & 0\\
1 & 0 & 0 & 0 & 1 & 0 & 1 & 1 & 0 & 0 & 0 & 0 & 1 & 0\\
1 & 0 & 0 & 0 & 0 & 1 & 0 & 1 & 1 & 0 & 0 & 0 & 0 & 0\\
1 & 0 & 0 & 0 & 0 & 1 & 1 & 0 & 1 & 1 & 0 & 0 & 0 & 1\\
1 & 0 & 0 & 0 & 0 & 0 & 1 & 1 & 0 & 1 & 0 & 0 & 0 & 0\\
1 & 1 & 0 & 0 & 0 & 0 & 0 & 1 & 1 & 0 & 0 & 0 & 0 & 0\\
1 & 1 & 0 & 0 & 0 & 0 & 0 & 0 & 0 & 0 & 0 & 0 & 0 & 0\\
1 & 0 & 0 & 1 & 0 & 0 & 0 & 0 & 0 & 0 & 0 & 0 & 0 & 0\\
1 & 0 & 0 & 0 & 0 & 1 & 0 & 0 & 0 & 0 & 0 & 0 & 0 & 0\\
1 & 0 & 0 & 0 & 0 & 0 & 0 & 1 & 0 & 0 & 0 & 0 & 0 & 0
\end{array}
\right].
\]
The graph $G$ is illustrated in Figure~\ref{fig:custom_graph_labeled_14} and has nullity one, {whose} kernel {is} spanned by the vector:
\[ x = (0, 0, 0, 0, 0, 0, 1, 0, 0, -1, 1, 0, -1, 0) ^T\]
This vector {satisfies} the row dependency $a^7 - a^{10} + a^{11} - a^{13} = 0 =\langle x,e\rangle.$ {Thus,} $0$ is not a main eigenvalue of $G$ and ${\bf e}$ lies in the row space of $A_G$. Thus $G$ satisfies the ACK conjecture without admitting a zero-sum subset $S:=\{v_7, v_{10}, v_{11}, v_{13}\}$ of size different from any vertex degree. 
\end{remark}
\section{Graphs in \texorpdfstring{$\cal{C}$}{C} satisfying ACK conjecture}\label{sec:graphsinC}
We introduce a new class of graphs, whose members satisfy the necessary conditions of Theorem \ref{neccon}, and the ACK conjecture, as well. This means that, despite the fact that these graphs are identified as potential counter examples, we prove that they are not, as such.

\subsection{Satellite graphs}
\begin{definition}\label{def:satellite}
For a positive integer $k$, let \(S_{2k+1}\) denote the graph on \(n = 2k + 1\) vertices, called satellite graphs, which are defined as follows: The vertex set \(V\) is partitioned into three disjoint sets:  
A single dominating vertex, \(v_{\text{dom}}\);  
A set of degree 4 vertices \(V_4 = \{u_1, u_2, \dots, u_k\}\) 
and a set of degree 2 vertices \(V_2 = \{w_1, w_2, \dots, w_k\}\), 
each consisting of \(k\) vertices. The edge set \(E\) is defined by:
\begin{itemize}
    \item The vertex \(v_{\text{dom}}\) is adjacent to every other vertex in the graph. so that \((v_{\text{dom}}, v) \in E\) for all \(v \in V_4 \cup V_2\).
    \item Each vertex in \(V_4\) is adjacent to two vertices in \(V_4\) and one vertex in \(V_2\).
    \item Each vertex in \(V_2\) is adjacent to one vertex in \(V_4\).
\end{itemize}
\end{definition}

It is clear that $k \geq 3.$ Thus, the smallest graph in this family is obtained when \(k=3\), with \(n = 2(3) + 1 = 7\) vertices. Henceforth, we will denote a satellite graph on \(2k+1\) vertices, by $S_{2k+1},$ also assuming that $k \geq 3.$ Let us list a few instances of such graphs.

\begin{figure}[h!]
\centering
\begin{tikzpicture}[scale=1] 
    \tikzset{
        v_dom/.style = {circle, fill=blue, draw=black, minimum size=4pt, inner sep=0pt},
        v_4/.style   = {circle, fill=red, draw=black, minimum size=4pt, inner sep=0pt},
        v_2/.style   = {circle, fill=black, draw=black, minimum size=4pt, inner sep=0pt}
    }

    \node[v_dom] (vdom) at (0,0) {};

    \foreach \i in {1,...,3} {
        \node[v_4] (u\i) at ({360/3 * (\i - 1) + 90} : 1.5cm) {};
    }

    \foreach \i in {1,...,3} {
        \node[v_2] (w\i) at ({360/3 * (\i - 1) + 100} : 2.5cm) {};
    }

    \foreach \i in {1,...,3} {
        \draw[thick] (vdom) -- (u\i);
        \draw[thick] (vdom) -- (w\i);
    }

    \draw[thick] (u1) -- (u2) -- (u3) -- (u1) --cycle;

    \foreach \i in {1,...,3} {
        \draw[thick] (u\i) -- (w\i);
    }

\end{tikzpicture}
\hspace{0.2cm}
\begin{tikzpicture}[scale=1] 

    \tikzset{
        v_dom/.style = {circle, fill=blue, draw=black, minimum size=4pt, inner sep=0pt},
        v_4/.style   = {circle, fill=red, draw=black, minimum size=4pt, inner sep=0pt},
        v_2/.style   = {circle, fill=black, draw=black, minimum size=4pt, inner sep=0pt}
    }

    \node[v_dom] (vdom) at (0,0) {};

    \foreach \i in {1,...,4} {
        \node[v_4] (u\i) at ({360/4 * (\i - 1) + 90} : 1.5cm) {};
    }

    \foreach \i in {1,...,4} {
        \node[v_2] (w\i) at ({360/4 * (\i - 1) + 100} : 2.5cm) {};
    }

    \foreach \i in {1,...,4} {
        \draw[thick] (vdom) -- (u\i);
        \draw[thick] (vdom) -- (w\i);
    }

    \draw[thick] (u1) -- (u2) -- (u3) -- (u4) -- (u1) --cycle;

    \foreach \i in {1,...,4} {
        \draw[thick] (u\i) -- (w\i);
    }

\end{tikzpicture}
\hspace{0.3cm}
\begin{tikzpicture}[scale=1] 

    \tikzset{
        v_dom/.style = {circle, fill=blue, draw=black, minimum size=4pt, inner sep=0pt},
        v_4/.style   = {circle, fill=red, draw=black, minimum size=4pt, inner sep=0pt},
        v_2/.style   = {circle, fill=black, draw=black, minimum size=4pt, inner sep=0pt}
    }

    \node[v_dom] (vdom) at (0,0) {};

    \foreach \i in {1,...,5} {
        \node[v_4] (u\i) at ({360/5 * (\i - 1) + 90} : 1.5cm) {};
    }

    \foreach \i in {1,...,5} {
        \node[v_2] (w\i) at ({360/5 * (\i - 1) + 100} : 2.5cm) {};
    }

    \foreach \i in {1,...,5} {
        \draw[thick] (vdom) -- (u\i);
        \draw[thick] (vdom) -- (w\i);
    }

    \draw[thick] (u1) -- (u2) -- (u3) -- (u4) -- (u5) -- (u1) --cycle;

    \foreach \i in {1,...,5} {
        \draw[thick] (u\i) -- (w\i);
    }

\end{tikzpicture}
\caption{Satellite graphs for $k=3,4$ and $5$.}
\label{fig:graph_g7}
\end{figure}
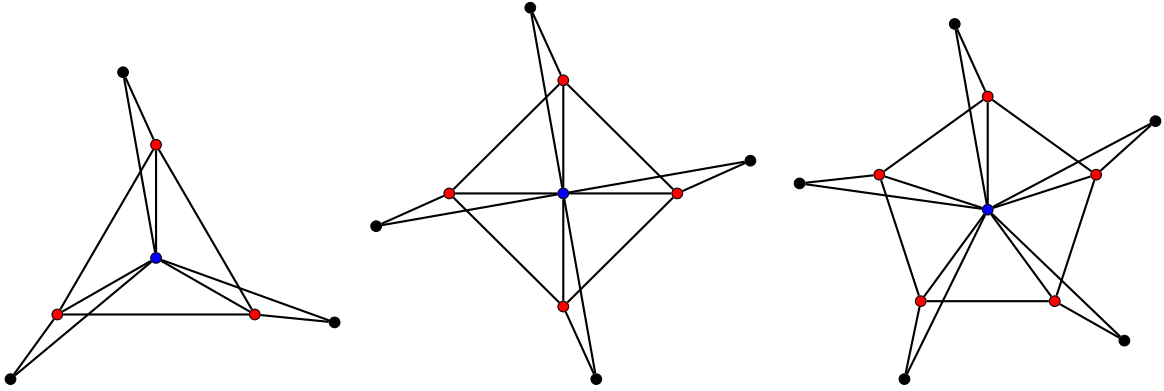
\noindent Next, we prove that satellite graphs are nut graphs.

\begin{theorem}\label{sat_nut}
Each $S_{2k+1}$ is a nut graph. 
\end{theorem}
\begin{proof}
Let \(x \in \mathbb{R}^{2k+1}\) be defined through the vertices of \(S_{2k+1}\) as:
\[
x_v = 
\begin{cases} 
+1 & \text{if } v \in V_4 \\
-1 & \text{if } v \in V_2 \\
-1 & \text{if } v = v_{\text{dom}}.
\end{cases}
\]
We claim that \(A(S_{2k+1})x = 0\), that is, For any vertex \(v \in V\), the \(v\)-th entry of the product \(A(S_{2k+1})x = \sum_{z \in N(v)} x_z\), is zero. If \(v = u_i \in V_4\), then \(N(u_i) = \{u_{i-1}, u_{i+1}, w_i, v_{\text{dom}}\}\) and he corresponding sum is:
    \[ \sum_{z \in N(u_i)} x_z = x_{u_{i-1}} + x_{u_{i+1}} + x_{w_i} + x_{v_{\text{dom}}} = (+1) + (+1) + (-1) + (-1) = 0. \]
Similarly, if \(v = w_i \in V_2\), then \(N(w_i) = \{u_i, v_{\text{dom}}\}\) and 
    \[ \sum_{z \in N(w_i)} x_z = x_{u_i} + x_{v_{\text{dom}}} = (+1) + (-1) = 0. \]
For dominating vertex $v_{\text{dom}}$ the neighborhood  \(N(v_{\text{dom}}) = V_4 \cup V_2\). Thus,     \[ \sum_{z \in N(v_{\text{dom}})} x_z = \sum_{j=1}^{n} x_{u_j} + \sum_{j=1}^{n} x_{w_j} = n \cdot (+1) + n \cdot (-1) = 0. \]
Next we prove that no nonzero vector in $N(A(S_{2k+1}))$ has a zero coordinate. Let $y \in N(A(S_{2k+1}))$, then for every vertex $v,$ we have $\sum_{z \in N(v)} y_z=0.$ For vertex \(w_i \in V_2, 1 \leq i \leq k,\) we have \(y_{u_i} + y_{v_{\text{dom}}} = 0,\) for all \(i=1, \dots, k\). Thus, 
$$y_{u_1} = y_{u_2} = \dots = y_{u_k} = -y_{v_{\text{dom}}}.$$ Letting this common value to be \(c\), we have \(y_{u_i} = c\) for all \(u_i \in V_4\) and \(y_{v_{\text{dom}}} = -c\). Next, for any vertex \(u_i \in V_4\), we have $y_{u_{i-1}} + y_{u_{i+1}} + y_{w_i} + y_{v_{\text{dom}}} = 0, ~1 \leq i \leq k.$ This simplifies to \(y_{w_i} = -c\), for all \(i=1, \dots, k\). Thus by Lemma \ref{nul1}, nullity of $A(S_{2k+1})$ is $1$.
\end{proof}

Now, we show that \(S_{2k+1}\) satisfies all the eight necessary conditions of Theorem \ref{neccon}. 

\begin{theorem}
$S_{2k+1} \in {\cal C}$.
\end{theorem}
\begin{proof}
The proof will be included in the revision to this version.
\end{proof}

Despite being a member of ${\cal C}$, \(S_{2k+1}\) is not a counter-example to the ACK conjecture. This is our next result.

\begin{theorem}\label{thm:ackholdsforsn}
\(S_{2k+1}\) satisfies the ACK conjecture.
\end{theorem}
\begin{proof}
The proof will be included in the next version.     
\end{proof}

\begin{remark}\label{comparison1}
It is pertinent to compare our results with those that were obtained, recently \cite{bera2022conjecture}, where the ACK conjecture is proved for diameter-$2$ graphs with no dominating vertex, and with exactly $2n-5$ edges, where $n$ is the number of vertices in $G$. In contrast, the satellite graphs constructed here have diameter $2$ and contain a dominating vertex.  Moreover, these graphs have nullity one and therefore do not belong to the other class of graphs considered in \cite{bera2022conjecture}, namely graphs of diameter at most $3$ and rank at most $5$. In our construction, each graph has rank $n-1$. Thus, the satellite graphs form a completely new family of graphs satisfying the ACK conjecture while lying in the class $\mathcal{C}$.
\end{remark}

\begin{definition}\label{def:evenverticesgraph} {Motivated by} the structure of satellite graphs, for integers $k= 4,5$, and $6$ we define graphs $E_{2k}$, with an even number of vertices (figure \ref{fig:custom_graph_labeled_8}).
    Here, $V$ is partitioned into four disjoint sets; a single dominating vertex, $v_{\text{dom}}$; a set of degree $4$ vertices, \(V_4 = \{u_1, u_2, \dots, u_k\}\); a set of $k-3$ degree \(2\) vertices, \(V_2 = \{w_1, w_2, \dots, w_{k-3}\}\) and two special vertices, $s_5$ and $s_6$ of degree $5$ and $6$, respectively. The edge set \(E\) is defined as follows:
    \begin{itemize}
    \item \((v_{\text{dom}}, v) \in E\), for all vertices $v \in V$.
    \item $s_5$ and $s_6$ are adjacent. The remaining edges must fulfill the degree requirements. Also,
    \begin{itemize}
    \item Each vertex in $V_4$ is adjacent to at least one other vertex in $V_4$.
    \item One vertex in $V_4$ is adjacent to both $s_5$ and $s_6$.
    \item One vertex in $V_2$ is adjacent to $s_5$.
\end{itemize}
\end{itemize}

\end{definition}

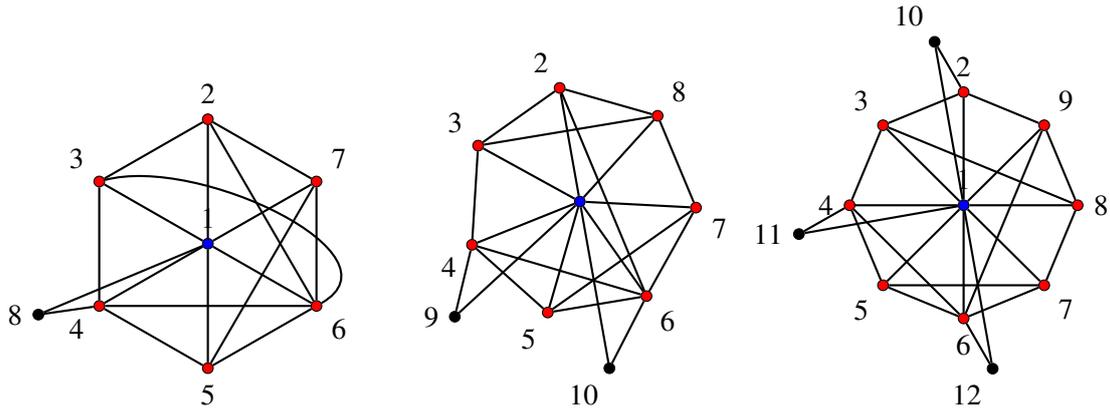
\begin{figure}[h!]
\centering
\begin{tikzpicture}[scale=1.1] 

    \tikzset{
        v_dom/.style = {circle, fill=blue, draw=black, minimum size=4pt, inner sep=0pt},
        v_4/.style   = {circle, fill=red, draw=black, minimum size=4pt, inner sep=0pt},
        v_2/.style   = {circle, fill=black, draw=black, minimum size=4pt, inner sep=0pt}
    }

    \node[v_dom, label=above:1] (vdom) at (0,0) {};

    \foreach \i in {1,...,6} {
        \pgfmathtruncatemacro{\labelnum}{\i+1} 
        \pgfmathtruncatemacro{\angle}{360/6 * (\i - 1) + 90} 
        \node[v_4, label=\angle:\labelnum] (u\i) at (\angle : 1.5cm) {};
    }

    \node[v_2, label={360/9 * (3 - 1) + 100}:8] (w3) at ({360/7 * (3 - 1) + 100} : 2.2cm) {};

    \foreach \i in {1,...,6} {
        \draw[thick] (vdom) -- (u\i);
    }
    \draw[thick] (vdom) -- (w3); 

    \draw[thick] (u1) -- (u2) -- (u3) -- (u4) -- (u5) -- (u6) -- (u1);

    \draw[thick] (u3) -- (w3);
    
    \draw[thick] (u3) -- (u5);
    \draw[thick] (u4) -- (u6);
    \draw[thick] (u1) -- (u5);
    \draw[thick] (u2) to[out=15, in=30] (u5); 

\end{tikzpicture}
\begin{tikzpicture}[scale=1.02] 

    \tikzset{
        v_dom/.style = {circle, fill=blue, draw=black, minimum size=4pt, inner sep=0pt},
        v_4/.style   = {circle, fill=red, draw=black, minimum size=4pt, inner sep=0pt},
        v_2/.style   = {circle, fill=black, draw=black, minimum size=4pt, inner sep=0pt}
    }

    \node[v_dom] (vdom) at (0,0) {};

    \foreach \i in {1,...,7} {
        \pgfmathtruncatemacro{\labelnum}{\i+1} 
        \pgfmathtruncatemacro{\angle}{360/7 * (\i - 1) + 100} 
        \node[v_4, label=\angle:\labelnum] (u\i) at (\angle : 1.5cm) {};
    }

    \node[v_2, label={360/9 * (3 - 1) + 100}:9] (w3) at ({360/7 * (3 - 1) + 120} : 2.2cm) {};
    \node[v_2, label={360/9 * (5 - 1) + 100}:10] (w5) at ({360/8 * (5 - 1) + 100} : 2.2cm) {};

    \foreach \i in {1,...,7} {
        \draw[thick] (vdom) -- (u\i);
    }
    \foreach \i in {3,5} {
        \draw[thick] (vdom) -- (w\i);
    }

    \draw[thick] (u1) -- (u2) -- (u3) -- (u4) -- (u5) -- (u6) -- (u7) -- (u1) --cycle;

    \foreach \i in {3,5} {
        \draw[thick] (u\i) -- (w\i);
    }
    
    \draw[thick] (u3) -- (u5);
    \draw[thick] (u4) -- (u6);
    \draw[thick] (u2) -- (u7);
    \draw[thick] (u1) -- (u5);
\end{tikzpicture}
\begin{tikzpicture}[scale=1] 

    \tikzset{
        v_dom/.style = {circle, fill=blue, draw=black, minimum size=4pt, inner sep=0pt},
        v_4/.style   = {circle, fill=red, draw=black, minimum size=4pt, inner sep=0pt},
        v_2/.style   = {circle, fill=black, draw=black, minimum size=4pt, inner sep=0pt}
    }

    \node[v_dom, label=above:1] (vdom) at (0,0) {};

    \foreach \i in {1,...,8} {
        \pgfmathtruncatemacro{\labelnum}{\i+1} 
        \pgfmathtruncatemacro{\angle}{360/8 * (\i - 1) + 90} 
        \node[v_4, label=\angle:\labelnum] (u\i) at (\angle : 1.5cm) {};
    }

    \node[v_2, label={360/9 * (1 - 1) + 100}:10] (w1) at ({360/8 * (1 - 1) + 100} : 2.2cm) {};
    \node[v_2, label={360/9 * (3 - 1) + 100}:11] (w3) at ({360/8 * (3 - 1) + 100} : 2.2cm) {};
    \node[v_2, label={360/9 * (5 - 1) + 100}:12] (w5) at ({360/8 * (5 - 1) + 100} : 2.2cm) {};

    \foreach \i in {1,...,8} {
        \draw[thick] (vdom) -- (u\i);
    }
    \foreach \i in {w1,w3,w5} {
        \draw[thick] (vdom) -- (\i);
    }

    \draw[thick] (u1) -- (u2) -- (u3) -- (u4) -- (u5) -- (u6) -- (u7) -- (u8) -- (u1);

    \draw[thick] (u1) -- (w1);
    \draw[thick] (u3) -- (w3);
    \draw[thick] (u5) -- (w5);
    
    \draw[thick] (u3) -- (u5);
    \draw[thick] (u4) -- (u6);
    \draw[thick] (u2) -- (u7);
    \draw[thick] (u5) -- (u8);

\end{tikzpicture}
\caption{Graph $E_{2k}$ for $k =4,5$ and $6$. }
\label{fig:custom_graph_labeled_8}
\end{figure}
\begin{theorem}
    $E_{2k},$ for $k=4,5,6$ are nut graphs. \label{subsec:Ennutgraphs}
\end{theorem} 
\begin{proof}
We identify a non-trivial vector $x\in \mathbb{R}^{2k}$ in $N(A(E_{2k}))$. 
For the 8-vertex graph, 
\[x = (1, 1, -1, -1, -1, -1, 1, 2)^T,\] 
{due to the fact that} 
    \[a^1 + a^2 + 2a^8 + a^7 - a^3 - a^4 - a^5 - a^6 = 0.\]\\
For the 10-vertex graph, 
\[x = (1, -1, -1, -1, -1, -1, 1, 1, 2, 1)^T,\] 
since
\[a^1 + a^7 + a^8 + 2a^9 + a^{10} - a^2 - a^3 - a^4 - a^5 - a^6 = 0,\]
while, for the 12-vertex graph, 
\[x = (1, -1, -1, -1, -1, -1, 1, 1, -1, 1, 2, 1)^T,\] 
{due} to the relation:

\[a^1 + a^7 + a^8 + 2a^{11} + a^{10} + a^{12} - a^2 - a^3 - a^4 - a^5 - a^6 - a^9 = 0.\]

    
\end{proof}

\begin{theorem}
{$E_{2k} \in \mathcal{C}$, $k=4,5,6$.}
\end{theorem}
\begin{proof}
{$E_{2k}$ for $k=4,5,6,$ by definition, are nut graphs, so that they are core graphs. Thus, the first condition of Theorem \ref{neccon} is satisfied.} {If $x$ is the vector as given in the proof of Theorem \ref{subsec:Ennutgraphs}, we then have $\langle x, {\bf e} \rangle \neq 0,$ so that the second condition of Theorem \ref{neccon} holds.} The construction of each $E_{2k}$ includes a dominating vertex, $v_{dom}$ and for any arbitrary vertex $u$ and any of its neighbors $w$, both $u$ and $w$ are adjacent to $v_{dom}$, forming the triangle $\{u, w, v_{dom}\}$. In fact, for every vertex $u$, each neighbor in $N(u)$ forms a triangle containing $u$, {showing} that conditions 3 and 4 of Theorem \ref{neccon} hold.

{Clearly, these graphs are not regular, are connected and are not bipartite. Finally, the dominating vertex guarantees that the path length between any two non-adjacent vertices is $2$, meaning the diameter is exactly 2. This shows that the last four conditions of Theorem \ref{neccon} are also satisfied, completing the proof.}
\end{proof}

\noindent The following result is a direct consequence of Corollary \ref{cor:zssnqdegree}.
\begin{corollary}\label{acknut}
The ACK conjecture holds for the nut graphs $E_{2k}$ for $k=4,5,6$.    
\end{corollary}
\begin{proof}
{Each $E_{2k}$ has no vertex of degree $3$ and has a zero-sum subset of size $3$.} 
\end{proof}

\begin{remark}\label{rem:k7failure}
The construction method for graphs as in Definition \ref{def:evenverticesgraph}, can be extended for integers $k \ge 7$ by following the same degree sequence. However, although the resulting graphs $E_{2k}$ may have a nullity of one, they need not be nut graphs. In particular, for $k=7$, the 14-vertex graph illustrated in Figure~\ref{fig:custom_graph_labeled_14}. This example shows that the construction does not, in general, produce graphs in the class $\mathcal{C}$ for $k \ge 7$. Nevertheless, for $k=8$ and $9$ by following the same construction and additionally adding one or two edges, one can obtain graphs that satisfy the ACK conjecture. These graphs, demonstrated in Figure \ref{fig:F18c}, however, do not lie in $\mathcal{C}$, since they are not nut graphs and $0$ is not a main eigenvalue. 
\end{remark}

\begin{remark}
Recall that, in \cite{bera2022conjecture}, the ACK conjecture is established for graphs of diameter at least four, as well as for certain graphs of diameter two or three whose structure admits pairs of adjacent vertices with disjoint neighbourhoods or arises from vertex-duplication procedures. In contrast, the graphs considered here have diameter two, contain a dominating vertex (a case not previously addressed), and every pair of adjacent vertices has a common neighbour. Consequently, the even-order graphs here, satisfy the ACK conjecture while not belonging to any of the graph families studied earlier.  
\end{remark}
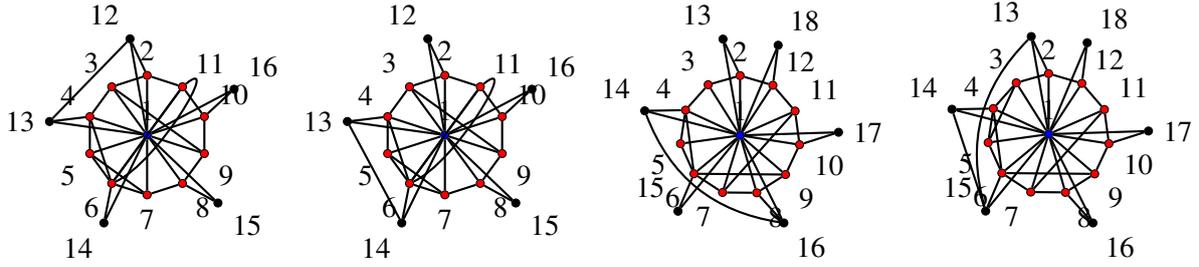
\begin{figure}[t]
\begin{tikzpicture}[scale=.72] 

    \tikzset{
    vertex/.style = {circle, draw=black, minimum size=3pt, inner sep=0pt},
    v_dom/.style  = {vertex, fill=blue},
    v_4/.style    = {vertex, fill=red},
    v_2/.style    = {vertex, fill=black}
}

    \node[v_dom, label=above:1] (vdom) at (0,0) {};

    \foreach \i in {1,...,10} {
        \pgfmathtruncatemacro{\labelnum}{\i+1} 
        \pgfmathtruncatemacro{\angle}{360/10 * (\i - 1) + 90} 
        \node[v_4, label=\angle:\labelnum] (u\i) at (\angle : 1.1cm) {};
    }

    \node[v_2, label={360/10 * (1 - 1) + 100}:12] (w1) at ({360/10 * (1 - 1) + 100} : 1.8cm) {};
    \node[v_2, label={360/9 * (3 - 1) + 100}:13] (w3) at ({360/10 * (3 - 1) + 100} : 1.8cm) {};
    \node[v_2, label={360/9 * (5 - 1) + 100}:14] (w5) at ({360/10 * (5 - 1) + 100} : 1.8cm) {};
    \node[v_2, label={360/9 * (7 - 1) + 100}:15] (w7) at ({360/10 * (7 - 1) + 100} : 1.8cm) {};
    \node[v_2, label={360/10 * (9 - 1) + 100}:16] (w9) at ({360/10 * (9 - 1) + 100} : 1.8cm) {};

    \foreach \i in {1,...,10} {
        \draw[thick] (vdom) -- (u\i);
    }
    \foreach \i in {w1,w3,w5,w7,w9} {
        \draw[thick] (vdom) -- (\i);
    }

    \draw[thick] (u1) -- (u2) -- (u3) -- (u4) -- (u5) -- (u6) -- (u7) -- (u8) -- (u9) --(u10) -- (u1);

    \draw[thick] (u1) -- (w1);
    \draw[thick] (u3) -- (w3);
    \draw[thick] (u5) -- (w5);
    \draw[thick] (u7) -- (w7);
    \draw[thick] (u9) -- (w9);
    
    \draw[thick] (u2) -- (u8);
    \draw[thick] (u3) -- (u5);
    \draw[thick] (u4) -- (u6);
    \draw[thick] (u5) -- (u10);
    \draw[thick] (u5) to[out=25, in=40] (u10); 
    \draw[thick] (w1) -- (w3);
\end{tikzpicture}
\begin{tikzpicture}[scale=.72] 
\tikzset{
    vertex/.style = {circle, draw=black, minimum size=3pt, inner sep=0pt},
    v_dom/.style  = {vertex, fill=blue},
    v_4/.style    = {vertex, fill=red},
    v_2/.style    = {vertex, fill=black}
}

    \node[v_dom, label=above:1] (vdom) at (0,0) {};

    \foreach \i in {1,...,10} {
        \pgfmathtruncatemacro{\labelnum}{\i+1} 
        \pgfmathtruncatemacro{\angle}{360/10 * (\i - 1) + 90} 
        \node[v_4, label=\angle:\labelnum] (u\i) at (\angle : 1.1cm) {};
    }

    \node[v_2, label={360/10 * (1 - 1) + 100}:12] (w1) at ({360/10 * (1 - 1) + 100} : 1.8cm) {};
    \node[v_2, label={360/9 * (3 - 1) + 100}:13] (w3) at ({360/10 * (3 - 1) + 100} : 1.8cm) {};
    \node[v_2, label={360/9 * (5 - 1) + 100}:14] (w5) at ({360/10 * (5 - 1) + 100} : 1.8cm) {};
    \node[v_2, label={360/9 * (7 - 1) + 100}:15] (w7) at ({360/10 * (7 - 1) + 100} : 1.8cm) {};
    \node[v_2, label={360/10 * (9 - 1) + 100}:16] (w9) at ({360/10 * (9 - 1) + 100} : 1.8cm) {};

    \foreach \i in {1,...,10} {
        \draw[thick] (vdom) -- (u\i);
    }
    \foreach \i in {w1,w3,w5,w7,w9} {
        \draw[thick] (vdom) -- (\i);
    }

    \draw[thick] (u1) -- (u2) -- (u3) -- (u4) -- (u5) -- (u6) -- (u7) -- (u8) -- (u9) --(u10) -- (u1);

    \draw[thick] (u1) -- (w1);
    \draw[thick] (u3) -- (w3);
    \draw[thick] (u5) -- (w5);
    \draw[thick] (u7) -- (w7);
    \draw[thick] (u9) -- (w9);
    
    \draw[thick] (u2) -- (u8);
    \draw[thick] (u3) -- (u5);
    \draw[thick] (u4) -- (u6);
    \draw[thick] (u5) -- (u10);
    \draw[thick] (u5) to[out=25, in=40] (u10); 
    \draw[thick] (w5) -- (w3);
\end{tikzpicture}
\begin{tikzpicture}[scale=.72] 

\tikzset{
    vertex/.style = {circle, draw=black, minimum size=3pt, inner sep=0pt},
    v_dom/.style  = {vertex, fill=blue},
    v_4/.style    = {vertex, fill=red},
    v_2/.style    = {vertex, fill=black}
}

    \node[v_dom, label=above:1] (vdom) at (0,0) {};

    \foreach \i in {1,...,11} {
        \pgfmathtruncatemacro{\labelnum}{\i+1} 
        \pgfmathtruncatemacro{\angle}{360/11 * (\i - 1) + 90} 
        \node[v_4, label=\angle:\labelnum] (u\i) at (\angle : 1.1cm) {};
    }

    \node[v_2, label={360/11 * (1 - 1) + 100}:13] (w1) at ({360/10 * (1 - 1) + 100} : 1.8cm) {};
    \node[v_2, label={360/11 * (3 - 1) + 90}:14] (w3) at ({360/11 * (3 - 1) + 100} : 1.8cm) {};
    \node[v_2, label={360/3 * (5 - 1) + 10}:15] (w5) at ({360/11 * (5 - 1) + 100} : 1.8cm) {};
    \node[v_2, label={360/11 * (7 - 1) + 100}:16] (w7) at ({360/11 * (7 - 1) + 100} : 1.8cm) {};
    \node[v_2, label={360/11 * (9 - 1) + 100}:17] (w9) at ({360/11 * (9 - 1) + 100} : 1.8cm) {};
    \node[v_2, label={360/11 * (11 - 1) + 100}:18] (w11) at ({360/11 * (11 - 1) + 100} : 1.8cm) {};

    \foreach \i in {1,...,11} {
        \draw[thick] (vdom) -- (u\i);
    }
    \foreach \i in {w1,w3,w5,w7,w9,w11} {
        \draw[thick] (vdom) -- (\i);
    }

    \draw[thick] (u1) -- (u2) -- (u3) -- (u4) -- (u5) -- (u6) -- (u7) -- (u8) -- (u9) --(u10) -- (u11) -- (u1);

    \draw[thick] (u1) -- (w1);
    \draw[thick] (u3) -- (w3);
    \draw[thick] (u5) -- (w5);
    \draw[thick] (u7) -- (w7);
    \draw[thick] (u9) -- (w9);
    \draw[thick] (u11) -- (w11);
    
    \draw[thick] (u5) -- (u8);
    \draw[thick] (u3) -- (u5);
    \draw[thick] (u6) -- (u10);
    \draw[thick] (w3) to[bend right=30] (w7);
\end{tikzpicture}
\begin{tikzpicture}[scale=.73] 
\tikzset{
    vertex/.style = {circle, draw=black, minimum size=3pt, inner sep=1pt},
    v_dom/.style  = {vertex, fill=blue},
    v_4/.style    = {vertex, fill=red},
    v_2/.style    = {vertex, fill=black}
}

    \node[v_dom, label=above:1] (vdom) at (0,0) {};

    \foreach \i in {1,...,11} {
        \pgfmathtruncatemacro{\labelnum}{\i+1} 
        \pgfmathtruncatemacro{\angle}{360/11 * (\i - 1) + 90} 
        \node[v_4, label=\angle:\labelnum] (u\i) at (\angle : 1.1cm) {};
    }

    \node[v_2, label={360/11 * (1 - 1) + 100}:13] (w1) at ({360/10 * (1 - 1) + 100} : 1.8cm) {};
    \node[v_2, label={360/11 * (3 - 1) + 90}:14] (w3) at ({360/11 * (3 - 1) + 100} : 1.8cm) {};
    \node[v_2, label={360/3 * (5 - 1) + 10}:15] (w5) at ({360/11 * (5 - 1) + 100} : 1.8cm) {};
    \node[v_2, label={360/11 * (7 - 1) + 100}:16] (w7) at ({360/11 * (7 - 1) + 100} : 1.8cm) {};
    \node[v_2, label={360/11 * (9 - 1) + 100}:17] (w9) at ({360/11 * (9 - 1) + 100} : 1.8cm) {};
    \node[v_2, label={360/11 * (11 - 1) + 100}:18] (w11) at ({360/11 * (11 - 1) + 100} : 1.8cm) {};

    \foreach \i in {1,...,11} {
        \draw[thick] (vdom) -- (u\i);
    }
    \foreach \i in {w1,w3,w5,w7,w9,w11} {
        \draw[thick] (vdom) -- (\i);
    }

    \draw[thick] (u1) -- (u2) -- (u3) -- (u4) -- (u5) -- (u6) -- (u7) -- (u8) -- (u9) --(u10) -- (u11) -- (u1);

    \draw[thick] (u1) -- (w1);
    \draw[thick] (u3) -- (w3);
    \draw[thick] (u5) -- (w5);
    \draw[thick] (u7) -- (w7);
    \draw[thick] (u9) -- (w9);
    \draw[thick] (u11) -- (w11);
    
    \draw[thick] (u5) -- (u8);
    \draw[thick] (u3) -- (u5);
    \draw[thick] (u2) -- (u4);
    \draw[thick] (u6) -- (u10);
    \draw[thick] (w1) to[bend right=30] (w5);
    \draw[thick] (w3) -- (w5);
\end{tikzpicture}
\caption{The first two graphs have degree sequence 
$\{2,2,2,3,3,4,4,4,4,4,4,4,4,5,6,15\}$ with dependence relation 
$a^7+a^9+a^{12}+a^{14}-a^4-a^{11}-2a^{15}=0$. 
The next two graphs have degree sequences 
$\{2,2,2,2,3,3,3,3,4,4,4,4,4,4,4,5,6,17\}$ and 
$\{2,2,2,3,3,4,4,4,4,4,4,4,4,4,4,5,6,17\}$, 
with dependence relations 
$a^3-a^5+a^{15}-a^{13}=0$ and 
$a^4-a^5+a^6-a^8+a^9-2a^3+a^{12}-a^{14}-a^{15}-a^{16}+2a^{18}=0$, respectively.
}
\label{fig:F18c}
\end{figure}

\section{Two graph operations and the ACK conjecture}\label{sec:graphopertn}

In this section, we consider the cartesian product of a graph with $K_2$. In the presence of a condition on the adjacency matrix of the given graph, the  new graph thus obtained satisfies the ACK conjecture.

Recall that the cartesian product $G_1 \square G_2$ of two graphs $G_1$ and $G_2$ is the graph with vertex set $V(G_1)\times V(G_2),$ where $(u_1,u_2)$ is adjacent to $(v_1,v_2)$ if and only if either
$u_1=v_1$ and $u_2\sim v_2$ in $G_2$, or $u_2=v_2$ and $u_1\sim v_1$ in $G_1$ \cite{brouwer2011spectra}.
In particular, when $G_1=K_2$ whose vertex set is labelled $\{0,1\}$, the graph $K_2 \square G_2$ consists of two copies of $G_2$, with an edge between $(0,u)$ and $(1,u)$ for each $u\in V(G_2)$. We use $\sigma(X)$ to denote the spectrum of the matrix $X$, namely its eigenvalues.
 
\begin{theorem}\label{thm:concart}
For a graph $H$, suppose that $1$ is a simple eigenvalue of $A_H$, and $-1$ is not an eigenvalue. Then $G = K_2 \square H$ satisfies the ACK conjecture.
\end{theorem}
\begin{proof}
The proof will be included in the revision to this version.
\end{proof}

\begin{remark}
The conclusion of Theorem \ref{thm:concart} also holds if the roles of the eigenvalues are interchanged, that is, if $-1$ is a simple eigenvalue of $A_H$ and $1$ is not. Also, $\operatorname{diam}(G)\le 3$ if and only if $\operatorname{diam}(H)\le 2$.
\end{remark}
\begin{remark}
Odd cycles $H:=C_{2k+1},~k\ge 4$ with $2k+1 \equiv 0 \pmod{3}$ provide a natural infinite family of graphs for which $-1$ is a simple eigenvalue of $A_H$ and $1 \notin \sigma(A_H)$. In fact, the spectrum of $C_n$ is explicitly given by 
$$
\sigma(C_n)=\left\{2\cos\!\left(\frac{2\pi j}{n}\right): j=0,1,\dots,n-1\right\},
$$
from which, the assertion follows. However, we do not know of any graph--theoretic characterization for these conditions. In particular, many other graphs, including irregular and asymmetric ones, also satisfy this condition. In the next example, one such  illustration is provided.
\end{remark}

\begin{example}\label{carprodex}
Let $H$ be the graph formed on vertices $\{1,2,3,4,5\}$ with the edge set
\[
E(H)=\{\{1,3\},\{1,4\},\{1,5\},\{2,3\},\{2,4\},\{3,4\}\}.
\]
The adjacency matrix of $H$ is
\[
A_H=\begin{pmatrix}
0 & 0 & 1 & 1 & 1 \\
0 & 0 & 1 & 1 & 0 \\
1 & 1 & 0 & 1 & 0 \\
1 & 1 & 1 & 0 & 0 \\
1 & 0 & 0 & 0 & 0
\end{pmatrix},
\]
with a simple eigenvalue $-1$ and $1$ is not an eigenvalue.  The kernel eigenvector for $\lambda = -1$ is $v = [0, 0, -1, 1, 0]^T$.
We construct the graph $G=K_2\square H$, as illistrated in Figure \ref{fig:conscart}. The adjacency matrix of $G$ has the block structure
\[
A_G=\begin{pmatrix}
A_H & I_5 \\
I_5 & A_H
\end{pmatrix},
\]
where $I_5$ is the $5\times 5$ identity matrix. By Theorem \ref{thm:concart}, $G$ satisfies the ACK conjecture.
\end{example}
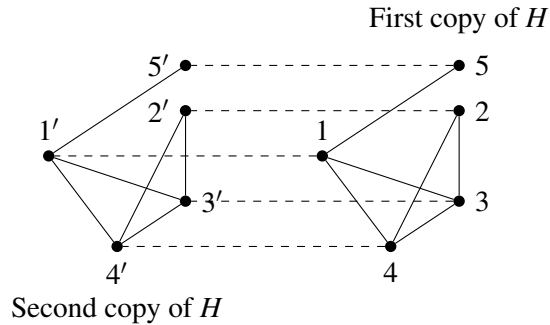
\begin{figure}[h!]
\centering
\begin{tikzpicture}[scale=1.2]
\tikzstyle{v} = [circle, fill=black, inner sep=1.5pt, outer sep=0pt]

\begin{scope}[xshift=3cm]
\node[v] (1a) at (0,0) [label=above:1] {};
\node[v] (2a) at (1.5,0.5) [label=right:2] {};
\node[v] (3a) at (1.5,-0.5) [label=right:3] {};
\node[v] (4a) at (0.75,-1) [label=below:4] {};

\node[v] (5a) at (1.5,1) [label=right:5] {};

\draw (1a) -- (3a);
\draw (1a) -- (4a);
\draw (2a) -- (3a);
\draw (2a) -- (4a);
\draw (3a) -- (4a);

\draw (1a) -- (5a);

\node at (1.5,1.5) {First copy of $H$};
\end{scope}

\begin{scope}[xshift=0cm]
\node[v] (1b) at (0,0) [label=above:1$'$] {};
\node[v] (2b) at (1.5,0.5) [label=left:2$'$] {};
\node[v] (3b) at (1.5,-0.5) [label=right:3$'$] {};
\node[v] (4b) at (0.75,-1) [label=below:4$'$] {};

\node[v] (5b) at (1.5,1) [label=left:5$'$] {};

\draw (1b) -- (3b);
\draw (1b) -- (4b);
\draw (2b) -- (3b);
\draw (2b) -- (4b);
\draw (3b) -- (4b);

\draw (1b) -- (5b);

\node at (0.75,-1.7) {Second copy of $H$};
\end{scope}

\draw[dashed] (1a) -- (1b);
\draw[dashed] (2a) -- (2b);
\draw[dashed] (3a) -- (3b);
\draw[dashed] (4a) -- (4b);
\draw[dashed] (5a) -- (5b);


\end{tikzpicture}
\caption{Graph $G = K_2 \square H$}
\label{fig:conscart}
\end{figure}

We obtain an immediate consequence of Theorem \ref{thm:concart}.

\begin{corollary}
Suppose that $1$ is a simple eigenvalue of $A_H$ with a corresponding full eigenvector $v$, and $-1$ is not an eigenvalue of $A_H$. Then $K_2 \square H$ satisfies the ACK conjecture.
\end{corollary}
\begin{proof}
By the hypotheses on $H$, $0$ is a simple eigenvalue of $A_G$, so that $G$ has nullity one. The corresponding kernel vector is $y = [v^T, -v^T]^T$. Since $v$ is a full vector, $y$ is also full. Therefore, $G$ is a nut graph completing the proof.
\end{proof}

In the next result, we consider the vertex addition operation.

\begin{theorem}\label{thm:single_vertex_addition} 
Let $G$ be a graph with a dominating vertex $1 \in V$. Suppose that the first coordinate of any non-zero vector in $N(A_G)$ is zero. Suppose there exist a non-empty, non-duplicate subset $S \subseteq V$ such that $\chi_S \perp N(A_G)$. Then (the diameter two graph) $H:=G + v_a$ obtained by adjoining a new vertex $v_a$ in such a way that $N(v_a) = S \cup \{1\}$, satisfies the ACK conjecture. 
\end{theorem}
\begin{proof} The proof will be included in the revision to this version.

\end{proof}




\begin{remark}
The construction in Theorem~\ref{thm:single_vertex_addition} readily extends to the simultaneous addition of multiple vertices. If $S_1,\dots,S_k \subseteq V(G)$ are non-empty, non-duplicate and disjoint subsets
such that $\chi_{S_i} \perp N(A(G))$ for each $i$, then adjoining vertices
$v_1,\dots,v_k$ with
\[
N(v_i) = S_i \cup \{1\}, \qquad i=1,\dots,k,
\]
produces a graph satisfying the ACK conjecture. The proof follows in same lines, since the kernel
equations decouple and force all new coordinates to vanish.
\end{remark}

\begin{example}
Let us start with graph $G$ with $16$ vertices given in the Figure \ref{fig:F18a} whose kernel eigenvector is
\[
x = (0,0,1,0,-1,0,0,0,\,0,0,\,0,\,0,\,-1,0,1,0,\,0,\,0)^T,\qquad A_Gx=0.
\]
Consider $S_1 = \{3,5\}$ and $S_2 = \{13,15\}$. 
Form $F$ by adding vertex $19$ and $20$ in such a way that  $N(19)=\{1,3,5\}$ and $N(20)=N(5)=\{1,13,15\}$.
\[
 y=(x,0,0)=(0,0,1,0,-1,0,0,0,\,0,0,\,0,\,0,\,-1,0,1,0,\,0,\,0,0,0)^T,\qquad A_Fy=0.
\]
\begin{figure}[ht]
\centering
\begin{tikzpicture}[scale=1.5] 

    \tikzset{
        v_dom/.style = {circle, fill=blue, draw=black, minimum size=4pt, inner sep=0pt},
        v_4/.style   = {circle, fill=red, draw=black, minimum size=4pt, inner sep=0pt},
        v_2/.style   = {circle, fill=black, draw=black, minimum size=4pt, inner sep=0pt}
    }

    \node[v_dom, label=above:1] (vdom) at (0,0) {};

    \foreach \i in {1,...,11} {
        \pgfmathtruncatemacro{\labelnum}{\i+1} 
        \pgfmathtruncatemacro{\angle}{360/11 * (\i - 1) + 90} 
        \node[v_4, label=\angle:\labelnum] (u\i) at (\angle : 1.1cm) {};
    }

    \node[v_2, label={360/11 * (1 - 1) + 100}:13] (w1) at ({360/10 * (1 - 1) + 100} : 1.8cm) {};
    \node[v_2, label={360/11 * (3 - 1) + 90}:14] (w3) at ({360/11 * (3 - 1) + 100} : 1.8cm) {};
    \node[v_2, label={360/3 * (5 - 1) + 10}:15] (w5) at ({360/11 * (5 - 1) + 100} : 1.8cm) {};
    \node[v_2, label={360/11 * (7 - 1) + 100}:16] (w7) at ({360/11 * (7 - 1) + 100} : 1.8cm) {};
    \node[v_2, label={360/11 * (9 - 1) + 100}:17] (w9) at ({360/11 * (9 - 1) + 100} : 1.8cm) {};
    \node[v_2, label={360/11 * (11 - 1) + 100}:18] (w11) at ({360/11 * (11 - 1) + 100} : 1.8cm) {};

    \foreach \i in {1,...,11} {
        \draw[thick] (vdom) -- (u\i);
    }
    \foreach \i in {w1,w3,w5,w7,w9,w11} {
        \draw[thick] (vdom) -- (\i);
    }

    \draw[thick] (u1) -- (u2) -- (u3) -- (u4) -- (u5) -- (u6) -- (u7) -- (u8) -- (u9) --(u10) -- (u11) -- (u1);

    \draw[thick] (u1) -- (w1);
    \draw[thick] (u3) -- (w3);
    \draw[thick] (u5) -- (w5);
    \draw[thick] (u7) -- (w7);
    \draw[thick] (u9) -- (w9);
    \draw[thick] (u11) -- (w11);
    
    \draw[thick] (u5) -- (u8);
    \draw[thick] (u3) -- (u5);
    \draw[thick] (u6) -- (u10);
    \draw[thick] (w3) to[bend right=50] (w7);
\end{tikzpicture}
\begin{tikzpicture}[scale=1.5] 

    \tikzset{
        v_dom/.style = {circle, fill=blue, draw=black, minimum size=4pt, inner sep=0pt},
        v_4/.style   = {circle, fill=red, draw=black, minimum size=4pt, inner sep=0pt},
        v_2/.style   = {circle, fill=black, draw=black, minimum size=4pt, inner sep=0pt},
        v_a/.style   = {circle, fill=white, draw=black, minimum size=4pt, inner sep=1pt},
    }

    \node[v_dom, label=above:1] (vdom) at (0,0) {};

    \foreach \i in {1,...,11} {
        \pgfmathtruncatemacro{\labelnum}{\i+1} 
        \pgfmathtruncatemacro{\angle}{360/11 * (\i - 1) + 90} 
        \node[v_4, label=\angle:\labelnum] (u\i) at (\angle : 1.1cm) {};
    }

    \node[v_2, label={360/11 * (1 - 1) + 100}:13] (w1) at ({360/11 * (1 - 1) + 100} : 1.8cm) {};
    \node[v_2, label={360/11 * (3 - 1) + 90}:14] (w3) at ({360/11 * (3 - 1) + 100} : 1.8cm) {};
    \node[v_2, label={360/3 * (5 - 1) + 10}:15] (w5) at ({360/11 * (5 - 1) + 100} : 1.8cm) {};
    \node[v_2, label={360/11 * (7 - 1) + 100}:16] (w7) at ({360/11 * (7 - 1) + 100} : 1.8cm) {};
    \node[v_2, label={360/11 * (9 - 1) + 100}:17] (w9) at ({360/11 * (9 - 1) + 100} : 1.8cm) {};
    \node[v_2, label={360/11 * (11 - 1) + 100}:18] (w11) at ({360/11 * (11 - 1) + 100} : 1.8cm) {};

\node[v_a, label={360/11 * (1 - 1) + 150}:19] (va) at ({360/10 * (3 - 1) + 100} : 2.7cm) {};
    \node[v_a, label={360/11 * (1 - 1) + 90}:20] (vb) at ({360/11 * (2 - 1) + 100} : 2.7cm) {};

    \foreach \i in {1,...,11} {
        \draw[thick] (vdom) -- (u\i);
    }
    \foreach \i in {w1,w3,w5,w7,w9,w11} {
        \draw[thick] (vdom) -- (\i);
    }
    \foreach \i in {va,vb} {
        \draw[thick] (vdom) -- (\i);
    }

    \draw[thick] (u1) -- (u2) -- (u3) -- (u4) -- (u5) -- (u6) -- (u7) -- (u8) -- (u9) --(u10) -- (u11) -- (u1);

    \draw[thick] (u1) -- (w1);
    \draw[thick] (u3) -- (w3);
    \draw[thick] (u5) -- (w5);
    \draw[thick] (u7) -- (w7);
    \draw[thick] (u9) -- (w9);
    \draw[thick] (u11) -- (w11);

    \draw[thick] (va) -- (u2);
    \draw[thick] (va) -- (u4);
    \draw[thick] (vb) -- (w1);
    \draw[thick] (vb) -- (w5);
    
    \draw[thick] (u5) -- (u8);
    \draw[thick] (u3) -- (u5);
    \draw[thick] (u6) -- (u10);
    \draw[thick] (w3) to[bend right=50] (w7);
\end{tikzpicture}
\caption{Dependence relation for both graphs is \(a^3-a^5+a^{15}-a^{13}= 0.\)}
\label{fig:F18a}
\end{figure}
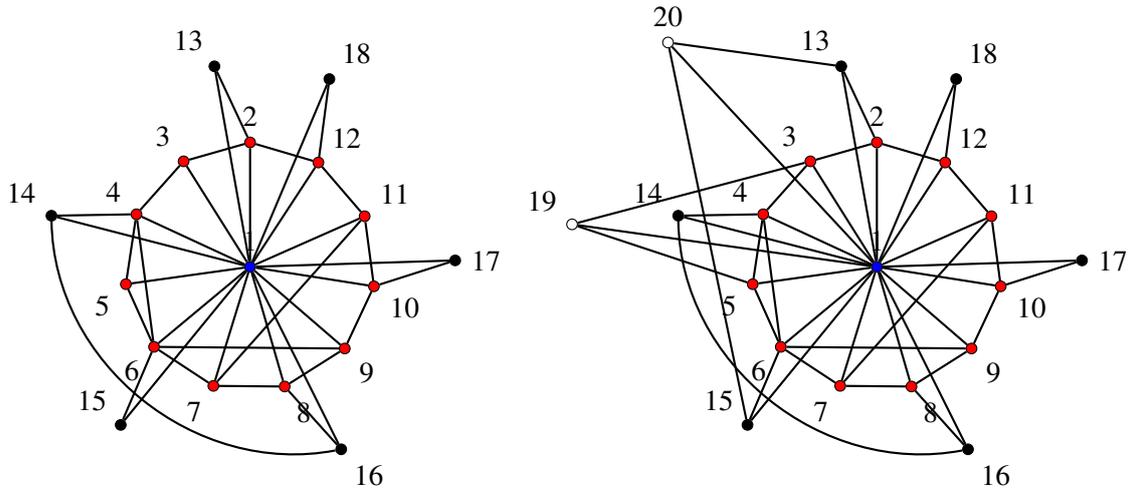

\end{example}

\section{Core graphs satisfying ACK conjecture}\label{sec:coreack}
In this concluding section, we obtain two classes of graphs satisfying the ACK conjecture. The first type is identified in the next result, while the second graph class is presented in Theorem \ref{secondcore}.

\begin{theorem}\label{firstcore}
Let $H$ be a simple graph such that $A_H$ is invertible. Set $B:=A_H^{-1}$. 
Let $S_1,\dots,S_k\subseteq V$ be nonempty subsets with characteristic vectors 
$c_1,\dots,c_k\in\{0,1\}^{|V|}$, respectively. Set $C:=[c_1,\ \cdots,\ c_k]$. 
Form a graph $G$ by adjoining new vertices $v_1,\dots,v_k$ to $H$ such that $N(v_i)=S_i$, for each $i$. Then the following hold:\\
(a) $G$ is a core graph (with nullity at least $k$) if and only if 
{$BC$ is a full matrix and $C^T B C = 0.$}\\
(b) If there exists an index $i$ for which $S_i$ is not equal to the neighbourhood of any vertex of $H$. Then \(G\) satisfies the ACK conjecture.
\end{theorem}
\begin{proof}
The proof will be included in the revision to this version.
\end{proof}

Let us give an illustration of Theorem \ref{firstcore}
for $k=2.$

\begin{example}
Let $H$ be the graph on six vertices $\{1,\dots,8\}$ with 
adjacency matrix
\[
A_H=
\begin{pmatrix}
0 & 0 & 1 & 0 & 1 & 1 & 0 & 0\\
0 & 0 & 0 & 1 & 1 & 1 & 0 & 1\\
1 & 0 & 0 & 0 & 1 & 0 & 1 & 1\\
0 & 1 & 0 & 0 & 1 & 1 & 1 & 0\\
1 & 1 & 1 & 1 & 0 & 1 & 0 & 1\\
1 & 1 & 0 & 1 & 1 & 0 & 0 & 1\\
0 & 0 & 1 & 1 & 0 & 0 & 0 & 0\\
0 & 1 & 1 & 0 & 1 & 1 & 0 & 0
\end{pmatrix}.
\]
This matrix is invertible with $\det A_H=-1$.
Choose the subsets
\[
S_1=\{6,8\},\qquad S_2=\{1,2\}.
\]
The corresponding characteristic vectors are
\[
c_1=(0,0,0,0,0,1,0,1)^T,\qquad 
c_2=(1,1,0,0,0,0,0,0)^T.
\]
\[
B c_1 = (-1,1,\tfrac12, -\tfrac12, \tfrac12, -1, -\tfrac12, 1)^T, \qquad
B c_2 = (1, -1, -\tfrac12, \tfrac12, \tfrac12, 1, -\tfrac12, -1)^T.
\]
Note that no coordinate of \(Bc_1\) or \(BC_2\) is zero. Form the matrix \(C=[c_1\ c_2]\). A direct check yields
\[
C^T B C \;=\; \begin{pmatrix} 0 & 0 \\[4pt] 0 & 0 \end{pmatrix}.
\]

Therefore the hypotheses of Theorem~\ref{firstcore} are satisfied. Adjoining two new vertices \(v_9\) and \(v_{10}\) to \(H\) with neighborhoods \(N(v_9)=S_1\) and \(N(v_{10})=S_2\) produces a graph \(G\) with nullity at least \(2\); moreover its kernel contains the independent vectors
\[
y_1=(Bc_1, -e_1)^T,\qquad y_2=(Bc_2, -e_2)^T,
\]
and \(G\) satisfies the ACK conjecture.
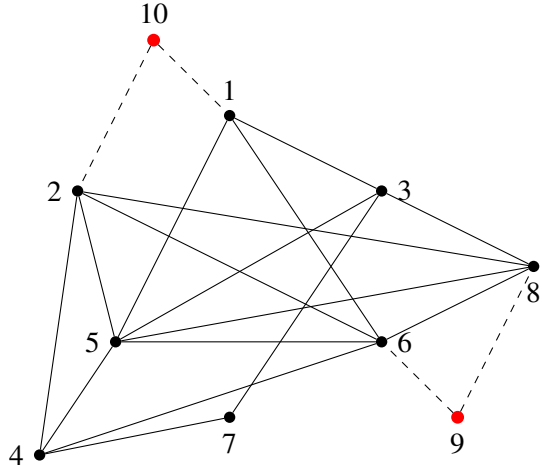
\begin{figure}[ht]
\centering
\begin{tikzpicture}
\tikzset{
        v_dom/.style = {circle, fill=blue, draw=black, minimum size=4pt, inner sep=0pt},
        v_4/.style   = {circle, fill=red, draw=black, minimum size=4pt, inner sep=0pt},
        v_2/.style   = {circle, fill=black, draw=black, minimum size=4pt, inner sep=0pt}
    }
\tikzstyle{v} = [circle, fill=black, inner sep=1.5pt, outer sep=0pt]
\tikzstyle{u} = [circle, fill=red, inner sep=1.7pt, outer sep=0pt]
    \node[v,label=above:$1$] (1) at (0,2) {};
    \node[v,label=left:$2$]  (2) at (-2,1) {};
    \node[v,label=right:$3$] (3) at (2,1) {};
    \node[v,label=left:$4$] (4) at (-2.5,-2.5) {};
    \node[v,label=left:$5$]  (5) at (-1.5,-1) {};
    \node[v,label=right:$6$] (6) at (2,-1) {};
    \node[v,label=below:$7$] (7) at (0,-2) {};
    \node[v,label=below:$8$] (8) at (4,0) {};

    \node[u,label=below:$9$] (9) at (3,-2) {};
    \node[u,label=above:$10$] (10) at (-1,3) {};
    


    \draw (1) -- (3);
    \draw (1) -- (5);
    \draw (1) -- (6);

    \draw (2) -- (4);
    \draw (2) -- (5);
    \draw (2) -- (6);
    \draw (2) -- (8);

    \draw (3) -- (5);
    \draw (3) -- (7);
    \draw (3) -- (8);

    \draw (4) -- (5);
    \draw (4) -- (6);
    \draw (4) -- (7);

    \draw (5) -- (6);
    \draw (5) -- (8);

    \draw (6) -- (8);


    \draw[dashed] (9) -- (6);
    \draw[dashed] (9) -- (8);

    \draw[dashed] (10) -- (1);
    \draw[dashed] (10) -- (2);

\end{tikzpicture}
\caption{Base graph \(H\) (vertices 1--8) with new vertices 9 and 10 attached: $N(9)=\{6,8\}$ and $N(10)=\{1,2\}$.}
\label{fig:multi-example}
\end{figure}
\end{example}

\begin{theorem}\label{secondcore}
Let $G$ be a nut graph. \\
(a) For $T = \{v_1, v_2, \dots, v_k\} \subseteq V$, let $F$ be the graph formed by duplicating each vertex $v_i \in T$, $m_i$ times, $m_i \geq 1$.Then, $F$ is a core graph.\\
(b) Let $G$ satisfy the ACK conjecture. Let $S$ be a non-empty, non-duplicate zero-sum subset of $V$ such that $S \cap T = \emptyset$, where $T$ is another subset of vertices. If $F$ is constructed from $T$ as above, then $F$ also satisfies the ACK conjecture.
\end{theorem}

\begin{proof}
The proof will be included in the revision to this version.
\end{proof}

In the next example, we demonstrate the complicated construction of the previous result.

\begin{example}
We start with the 7-vertex nut graph $G$ of Figure \ref{fig:nut_construction} whose kernel eigenvector ($A_Gx=0$) is
\[
x=(1,\,1,\,-1,\,-1,\,-1,\,1,\,-1)^T.
\]
Form $F$ by duplicating vertex $1$ once and vertex $5$ twice; name the duplicates $1'$ and $5'$, with $N(1')=N(1)=\{2,3,4,6\}$ and $N(5')=N(5)=\{2,7\}$.
Define
\[
y=(x_1,x_1,x_2,x_3,x_4,x_5,x_5,x_5,x_6,x_7)^T=(1,\,1,\,1,\,-1,\,-1,\,-1,\,-1,\,-1,\,1,\,-1)^T.
\]
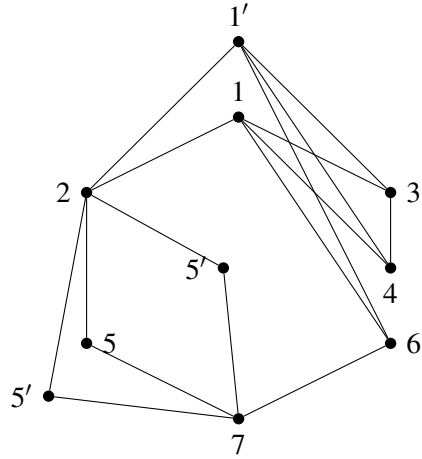
\begin{figure}[ht]
\centering
\begin{tikzpicture}
    \tikzstyle{v} = [circle, fill=black, inner sep=1.5pt, outer sep=0pt]

    \node[v,label=above:$1'$] (1') at (0,1) {};
    \node[v,label=above:$1$] (1) at (0,0) {};
    \node[v,label=left:$2$]  (2) at (-2,-1) {};
    \node[v,label=right:$3$] (3) at (2,-1) {};
    \node[v,label=below:$4$] (4) at (2,-2) {};
    \node[v,label=right:$5$]  (5) at (-2,-3) {};
    \node[v,label=left:$5'$]  (5') at (-0.2,-2) {};
    \node[v,label=left:$5'$]  (5'') at (-2.5,-3.7) {};
    \node[v,label=right:$6$] (6) at (2,-3) {};
    \node[v,label=below:$7$] (7) at (0,-4) {};

    \draw (1)--(2);
    \draw (1)--(3);
    \draw (1)--(4);
    \draw (1)--(6);

    \draw (1')--(2);
    \draw (1')--(3);
    \draw (1')--(4);
    \draw (1')--(6);
    
    \draw (2)--(5);
    \draw (2)--(5');
    \draw (2)--(5'');
    \draw (3)--(4);

    \draw (5)--(7);
    \draw (5')--(7);
    \draw (5'')--(7);
    \draw (6)--(7);
\end{tikzpicture}
\caption{A 10-vertex core graph constructed from a 7-vertex nut graph by duplicating vertices $1$ and $5$.}
\label{fig:nuttocore_example}
\end{figure}
Since $N(1')=N(1)$ and $N(5')=N(5)$,
we have $A_Fy=0$. The difference vectors are given by 
$$
d_{1,1}=(1,-1,0,0,0,0,0,0,0,0)^T,$$
$$d_{5,1}=(0,0,0,0,0,1,-1,0,0,0)^T$$
and $$d_{5,2}=(0,0,0,0,0,1,0,-1,0,0)^T.$$

It may be verified that $\{y, d_{11},d_{51}, d_{52}\}$ is linearly independent and that they lie in $N(A_F)$.
Moreover, every vertex of $F$ has a nonzero coordinate in at least one of these vectors, and so $F$ is a core graph.
\end{example}

\section*{Concluding Remarks}
In this work, we demonstrate that the necessary conditions for potential counter-examples to the Akbari--Cameron--Khosrovshahi conjecture identified by Sciriha et al. are not sufficient.
In particular, using kernel-vector-based zero-sum subsets and explicit graph constructions, we showed that for every $n \ge 7$ there exists a graph of order $n$ lying in the class $\mathcal{C}$ of potential counter-examples, that nevertheless satisfies the ACK conjecture. While a complete classification of graphs in $\mathcal{C}$ remains open, the methods developed here further narrow the range of possible counter-examples and provide new constructive and structural tools that support the validity of the conjecture.

\section*{Acknowledgements}
The authors gratefully acknowledge K. Kranthi Priya and Aditi Howlader for their valuable discussions and insightful comments that contributed to this work.


\end{document}